\documentclass[12pt]{article}
\usepackage{amsmath,amsfonts,amssymb,amsthm}
\usepackage{longtable}

\def\qed{{$\square$}}

\def\labtt#1{\label {#1} }
\def\refpp#1{\ref {#1}}

\def\a{\alpha}
\def\b{\beta}

\def\CC{{\mathbb C}}
\def\FF{{\mathbb F}}
\def\QQ{{\mathbb Q}}

\def\ZZ{{\mathbb Z}}
\def\NN{{\mathbb N}}

\def\la{\langle}
\def\ra{\rangle}
\def\<{\langle}
\def\>{\rangle}

\def\bs{\it}            

\def\dim{{\bs dim}}

\def\Sym{{\bs Sym}}

\def\sgn{{\rm sgn}}

\def\vac{\hbox{\bf 1}} 

\newcommand{\x}{\mathbf{x}}
\newcommand{\y}{\mathbf{y}}

\newtheorem{mainthm}{Main Theorem}
\newtheorem{thm}{Theorem}[section]
\newtheorem{prop}[thm]{Proposition}
\newtheorem{lem}[thm]{Lemma}
\newtheorem{rem}[thm]{Remark}

\newtheorem{de}[thm]{Definition}

\newtheorem{nota}[thm]{Notation}
\newtheorem{ex}[thm]{Example}

\begin{document}
\begin{center}

{\Large \bf
Lattice vertex algebras over a field of positive characteristic: degenerate cases
}

\vspace{5mm}
Robert L.~Griess, Jr. 
\\[0pt]
Department of Mathematics\\[0pt] University of Michigan\\[0pt]
Ann Arbor, MI 48109  USA  \\[0pt]
{\tt rlg@umich.edu}\\[0pt]
\vskip 0.3 cm

Ching Hung Lam
\\[0pt]
Institute of Mathematics \\[0pt]
Academia Sinica\\[0pt]
Taipei 10617, Taiwan\\[0pt]
{\tt chlam@math.sinica.edu.tw}\\[0pt]
\vskip 0.3cm
\end{center}

\medskip

\begin{abstract}
A vertex operator algebra $V_L$ associated with a positive definite even lattice $L$ has a standard integral form, which we denote it by $V_{L,\ZZ} $.   If $F$ is a field of characteristic $p>0$, it is known that $V_{L,F}:= F\otimes_\ZZ V_{L,\ZZ}$, a vertex algebra  over $F$, is  simple  if and only if $(p, \det(L))=1$. 

In this article, we study $V_{L,F}$ when the characteristic of $F$ divides $\det(L)$. We determine the radical $\mathrm{Rad}$  of the invariant bilinear form on $V_{L,F}$, show that it is the unique maximal ideal and study the quotient vertex algebra $V_{L,F}/\mathrm{Rad}$.    
\end{abstract}

\newpage 
\tableofcontents

\medskip

\begin{center}
{\bf \large Notation and Terminology}

{\small 
\begin{longtable}{|c|c|c|}
   \hline
\bf{Notation}& \bf{Explanation} & \bf{Examples in text}  \cr
   \hline\hline
$\delta$, $\delta_k$ & $\delta=\delta_k=(k-1, \dots, 1, 0)$ & Definition \ref{deslambda} \cr \hline 
$\Delta(x_1, \dots, x_k)$ & $\Delta(x_1, \dots, x_k) = \sum_{w\in Sym_k} \sgn(w) \x^{w(\delta)}$ & Equation \eqref{eq3}\cr \hline 
$h_{\lambda}(\alpha)$ &  $h_{\lambda}(\alpha)= s_{\alpha, \lambda_1}\cdots s_{\alpha, \lambda_k} $, $\lambda=(\lambda_1, \dots, \lambda_k)\in \ZZ$ & Remark \ref{rem3.5} \cr \hline 
$\langle\ , \ \rangle $, $\langle\ , \ \rangle_F$ & the symmetric invariant bilinear forms  & Notation \ref{standardif} \cr 
& on 
$ V_{L,\ZZ}$ and $V_{L,F}$ & \cr \hline    
$\mathcal{P}$ & the set of all partitions of positive integers & Section \ref{sec3.1}\cr \hline  
$\mathcal{P}_k$& the set of all partitions of length $\leq k$ & Section \ref{sec3.1}\cr \hline   
 $\mathrm{Rad}$& the radical of $\langle\ , \ \rangle_F$ on $V_{L,F}$&  Theorem  \ref{main} \cr \hline 
   $s_{\alpha, n}$ & the coefficient of $z^n$ in $\exp( \sum_{n>0} \frac{\alpha(-n)}{n} z^n)$&  Definition \ref{de_san}\cr
& i.e., $\exp\left( \sum_{n>0} \frac{\alpha(-n)}{n} z^n\right) = \sum_{n\geq 0} s_{\alpha,n}z^n$ & \cr \hline
$s_{\lambda} (\alpha)$& $s_{\lambda} (\alpha)= \det(s_{\alpha, \lambda_i -i+j})_{1\leq i,j\leq k}$ for $\lambda\in\mathcal{P}, \alpha\in L$ &  Definition \ref{deslambda} \cr \hline 
$s^{(s)}_\lambda( \alpha)$ & $s^{(s)}_\lambda( \alpha) =  \det(s_{\alpha, p^{s}(\lambda_i -i+j)})_{1\leq i,j\leq k}$, $p$ a prime & Definition \ref{de4.7} \cr \hline 
VA, VOA & vertex algebra, vertex operator algebra & Definitions \refpp{defva},\refpp{defvoa} \cr
\hline
$V_L$ & the lattice VOA associated with $L$ &  Introduction  \cr\hline
$V_{L,\ZZ}$  & the standard integral form of $V_L$ &  Notation \refpp{standardif}\cr
\hline
$V_{L,F}$  & the VOA $F \otimes_{\ZZ} V_{L, \ZZ}$ for a field $F$ &  Introduction \cr
\hline
\end{longtable}}
\end{center}

\section{Introduction}
Through out this article, we use the standard notation for the lattice vertex operator algebra $V_L = M(1) \otimes \CC\{L\}$ associated with a positive definite even lattice $L$  \cite{FLM}. In particular,
${\mathfrak h}=\CC\otimes_{\ZZ} L$  and  $\hat {\mathfrak{h}}={\mathfrak h}\otimes \CC[t,t^{-1}]\oplus \CC k$ is the corresponding affine Lie
algebra. Here $M(1)=\CC[\a_i(n)|1\leq i\leq d, n<0]$ is the polynomial ring in $\a_i(n)$,  where $\{\a_1,
\dots,\a_d\}$  is basis of $\mathfrak{h}$ and $\a(n)=\a\otimes t^n,$ for $\alpha\in L$ and $n\in \ZZ$.  Also, $\CC\{L\}=\mathrm{span}
\{e^{\beta}\mid \beta\in L\}$ is the twisted group algebra of the additive group
$L$ such that $e^\b e^\a=(-1)^{\la \a, \b\ra} e^\a e^\b$ for any $\a, \b\in L$.
The vacuum vector $\vac$  of $V_L$ is $1\otimes e^0$.  For
the explicit definition of the corresponding vertex operators, we shall refer to
\cite{FLM}.

The vertex operator algebra $V_L$ has a standard integral form \cite{ivoa,Pro} given by 
the $\ZZ$-span of
\[
\{ s_{\alpha_1,n_1}s_{\alpha_2,n_2} \cdots s_{\alpha_k,n_k}\otimes e^\alpha \mid \alpha_i, \alpha\in L , n_i >0\}, 
\]
where $s_{\alpha, n}$ is the coefficient of $z^n$ in
$E^-(-\alpha, z) = \exp( \sum_{n>0} \frac{\alpha(-n)}{n} z^n)$, i.e,
\[
\exp\left( \sum_{n>0} \frac{\alpha(-n)}{n} z^n\right) = \sum_{n\geq 0} s_{\alpha,n}z^n.
\]
We denote this integral form by $V_{L,\ZZ} $.    
 
Let $F$ be a field of characteristic $p>0$. Then  $V_{L,F}:= F\otimes_\ZZ V_{L,\ZZ}$ is a vertex algebra over $F$ (cf. Definition \ref{defva}).  It is known \cite{Mu} that $V_{L,F}$ is a simple vertex operator algebra over $F$ if and only if $(p, \det(L))=1$.  

In this article, we study the  vertex algebra $V_{L,F}$ when the characteristic of $F$ divides $\det(L)$. We will determine the unique maximal ideal $J$ of $V_{L,F}$ and study its simple quotient
 $V_{L,F}/J$. 
 
The main idea is to consider the symmetric invariant bilinear forms $\langle\ , \ \rangle $ on 
$ V_{L,\ZZ}$ and $V_{L,F}$.  
It turns out that there is an explicit formula for calculating the values of 
$\langle s_{\alpha_1,n_1}\cdots s_{\alpha_k,n_k}\otimes e^\alpha, s_{\beta_1,m_1} \cdots s_{\beta_\ell,m_\ell}\otimes e^{\beta}\rangle$ (cf. Lemma \ref{innerprod}).  
We will  describe the radical of $\langle\ , \ \rangle_F$ by using this formula.

The Smith invariant theory for the containment  $L\subseteq L^*$ gives a basis $\{x_1, \dots, x_d\}$ of the dual lattice $L^*$ and a set of positive integers $a_1| a_2|\cdots| a_d$ such that 
$\{a_1x_1, a_2x_2, \dots, a_dx_d\}$ is a basis of $L$.

Let $\alpha_i=a_ix_i$ and let $s_i$ be the largest non-negative integer such that $a_i=p^{s_i} r_i$.  The key observation is that $
\langle s_{\alpha_i,n}, V_{L, F} \rangle =0  $   
for any $n\in \ZZ_{>0}$ such that ${p^{s_i}\not|\,n}$ (cf. Lemma \ref{sRi}).   
It turns out that the radical of $\langle\ , \ \rangle$ on $V_{L,F}$ is given by 
$K\otimes F\{L\}$, where $K$ is the ideal of $M(1)_F$ generated by 
$$\{ s_{\alpha_i,n}\mid n\in \ZZ_+,\ p^{s_i}\not|\, n, i=1,\dots,d\}.$$

Note that $\mathrm{Rad}$, the radical of  of $\langle\ , \ \rangle_F$,  is an ideal of $V_{L,F}$ since $\langle \ , \ \rangle_F$ is an invariant bilinear form and $V_{L,F}/\mathrm{Rad}$ is a simple vertex algebra \cite{Li}.  Therefore, the unique maximal ideal $J$ of $V_{L,F}$ is the same as $\mathrm{Rad}$.  

The main result of this article is as follows (cf. Theorem \ref{main}). 

\begin{mainthm}\label{mainthm}
Let $L$ be a positive definite even lattice and let $F$ be a field of characteristics $p>0$. 
Let 
\[
U_F = F[s_{\alpha_i, m}\mid    m\in \ZZ_+,  p^{s_i}| m \text{ for all }i=1, \dots, n] \otimes F\{L\}.  
\]
Then  the bilinear form $\langle \ ,  \ \rangle$ on $U_F$ is non-degenerate.  Thus, the radical $\mathrm{Rad}$ of 
$\langle \  ,  \ \rangle$ is given by $K\otimes F\{L\}$  and 
\[
V_{L,F}/ \mathrm{Rad}\cong U_F= F[ s_{\alpha_i, p^{s_i} n}\mid  n\in \ZZ_+, i=1, \dots,d] \otimes F\{L\} 
\]
as vector spaces. 
\end{mainthm}

The organization of the article is as follows. In Section 2, we review some basic definitions for vertex algebras over commutative rings and their integral forms. In Section 3, we recall some results from \cite{borcherds,ivoa,Pro} about the standard integral form 
of lattice vertex operator algebras. We also review  some formulas for calculating the values of the symmetric invariant bilinear form (cf. Lemmas \ref{innerprod} and \ref{slambda1}).   
In Section 4, we study and determine the radical of the symmetric invariant bilinear form $\langle\ ,\ \rangle$ on $V_{L,F}$ if the characteristic of $F$ divides $\det(L)$.
Some explicit examples are also discussed.

\section{Vertex algebras over commutative rings}

We begin by reviewing some basic concepts.

\begin{de}[\cite{borcherds,BR,FLM}]\labtt{defva}
Let $R$ be a commutative ring with identity.  A vertex algebra (VA) over $R$  is a $R$-module $V$ equipped with a linear map
\begin{eqnarray*}
Y(\ ,z):V &\longrightarrow &\mathrm{End}\ V\left[ \left[ z,z^{-1}\right]
\right] \\
v &\longrightarrow &Y(v,z)=\sum_{i\in \mathbb{Z}}v_{i}z^{-i-1}
\end{eqnarray*}
and a linear map $T:V \to V$ such that the following conditions hold:

\begin{enumerate}
\item  \textbf{(Vacuum condition)} there is a vector $\vac$ such that
\[
Y(\vac,z)=id_V = id_V \cdot z^0;
\]

\item  \textbf{(Creation property)} $Y(a,z)\cdot \vac\in V\left[ \left[ z\right] \right]
$ and $a_{-1}\cdot \vac =a$ for any $a\in V$;

\item  for any $a,b\in V$,  $a_{n}b=0$  for  $n$  sufficiently large;

\item  \textbf{(Translation  property)}   $T\vac =0$ and
\[
[T, Y( a,z)] =Y(Ta, z)= \frac{d}{dz}Y(a,z);
\]

\item  \textbf{(Borcherds Identity)} for any $a,b,c\in V$ and  $m,n, q\in \mathbb{Z}$,
\[
\begin{split}
&\sum_{i\geq 0} \binom{m}{i} (a_{q+i}b)_{m+n-i} c\\
 = & \sum_{i\geq 0} (-1)^ i  \binom{q}{ i} \left( a_{m+q-i} (b_{n+i}c) -  (-1)^q  b_{n+q-i} (a_{m+i} c) \right).
 \end{split}
 \]
\end{enumerate}

A vertex algebra $V$ is $\mathbb{Z}$-graded if there is a nontrivial direct sum decomposition  $ V=\oplus _{n\in \ZZ}  V_{n}$ such that  for all $m,n$, whenever   $a\in V_n$, $b\in V_m$,  then $a_k b\in V_{n+m-k-1}$ for any integer $k$.
\end{de}

\begin{de}\labtt{defvoa}
A $\ZZ$-graded VA  $V=\oplus _{n\in \ZZ}  V_{n}$ over $R$ is said to be a  \emph{vertex operator algebra (VOA)}  if there is an element $\omega \in V$ such that the operators $L(n)=\omega_{n+1}, n\in \ZZ,$ satisfy the Virasoro relation
\[
[L(m) , L (n)] = (m - n) L (m + n) + \binom{m+1}3 \delta_{m+n,0} c'
\]
for  any $m, n\in \ZZ$, where $c' \in R$.   We call  $2c'$ the rank or the central charge of $V$.   Moreover,
\begin{enumerate}
\item  the $\ZZ$-grading of $V$ is compatible with the action of $L(0)$, i.e.,
\[
L (0) v = nv \quad \text{ for any }v \in V_n;
\]
(note:  $n\in \ZZ$, whereas the eigenvalues of $L(0)$ are contained in $\ZZ\cdot 1_R \subseteq   R$)

\item $\dim V_n< \infty $ for any $n\in \ZZ$ and $V_n=0$ for $n$ sufficiently negative;

\item $T=L(-1)$. In particular,  $\frac{d} {dx} Y (v, z) = Y (L ( -1) v, z)$ for any $v\in V$.
\end{enumerate}
\end{de}

\begin{rem}
 Our definition of rank is worded  differently from the usual one since the ring $R$ may not contain $\frac 12$.
\end{rem}

\begin{de}\labtt{defif} Suppose that $V$ is a VOA (over the complex numbers) with a
nondegenerate symmetric invariant bilinear form.

 An integral VOA form  for $V$ is an abelian subgroup $J$ of $(V, +)$ such that $J$ is a VA over
$\mathbb{Z}$; there exists a positive integer $s$ so that $s\omega\in J$; for each $n$, $J_n := J\cap V_n$
is an integral form of $V_n$; $\langle J, J\rangle \subseteq \mathbb{Q}$. Since $J$ is a VA, $\mathbf{1}\in J$. For each degree $n$, $J_n$ has finite rank, whence there is an integer $d(n) > 0$ so that
$d(n)\cdot \langle J_n, J_n\rangle \subseteq \mathbb{Z}$.
\end{de}

\begin{rem}
An integral form $J$ of a VOA over $\CC$ will be a VA over $\mathbb{Z}$. If $R$ is a commutative ring, then $J\otimes_\ZZ
R $ is a VA over $R$.
\end{rem}

\section{Integral forms for lattice vertex operator algebras}\label{sec:2}

In \cite{borcherds,ivoa,Pro}, an integral form of the lattice VOA $V_L$ has been investigated. We first recall their results.

\begin{de}\label{de_san}
For any $\alpha \in L$, let $s_{\alpha, n}$ be the coefficient of $z^n$ in\\
$E^-(-\alpha, z) = \exp( \sum_{n>0} \frac{\alpha(-n)}{n} z^n)$, i.e,
\[
\exp\left( \sum_{n>0} \frac{\alpha(-n)}{n} z^n\right) = \sum_{n\geq 0} s_{\alpha,n}z^n.
\]
\end{de}

\begin{nota}\labtt{standardif}
By \cite[Theorem 3.3]{ivoa} (see also \cite[Section 4.4]{Pro}),
the $\ZZ$-span of
\[
\{ s_{\alpha_1,n_1}s_{\alpha_2,n_2} \cdots s_{\alpha_k,n_k}\otimes e^\alpha \mid \alpha_i, \alpha\in L , n_i >0\}.
\]
is an integral form of $V_L$.  We call this integral form the standard integral form for $V_L$ and denote it by $V_{L,\ZZ}$.
\end{nota}

There is a unique nondegenerate symmetric invariant bilinear form $\langle \cdot, \cdot\rangle$ on $V_L$ such that 
\[
\begin{split}
&\langle e^\alpha, e^\beta \rangle =\delta_{\alpha+\beta,0}, \\
&\langle \alpha(n)u, v\rangle = -\langle u, \alpha(-n)v\rangle, 
\end{split}
\]
for any $u,v\in V_L$, $\alpha, \beta\in L$ and $n\in \ZZ$ \cite{FLM}. 

\medskip

The bilinear form $\langle \cdot, \cdot\rangle|_{V_{L,\ZZ}}$ was studied in \cite{ivoa}. 
Recall from \cite[page 93]{FLM} that 
\[ 
\left[ \sum_{m\in \ZZ_+} \frac{-\alpha(m)w^m}{m}, 
\sum_{n\in \ZZ_+} \frac{\beta(-n)x^{n}}{n}\right] = -\langle \alpha, \beta\rangle
\sum_{n\in \ZZ_+}\frac{(wx)^n}{n} = \langle \alpha, \beta\rangle \log(1-wx)
\]
and hence 
\[
\begin{split}
&\sum\langle s_{\alpha_1,m_1} \dots s_{\alpha_k,m_k}e^\alpha, s_{\beta_1,n_1} \dots s_{\beta_\ell,n_\ell}e^\beta \rangle 
w_1^{m_1}\cdots w_k^{m_k} x_1^{n_1} \cdots x_\ell^{n_{\ell}}\\
= &
\langle \exp\left(
\sum_{i=1}^k \sum_{n\geq 1} \frac{\alpha_i(-n)}{n} w_i^n \right) e^\alpha, 
\exp\left(\sum_{j=1}^\ell \sum_{n\geq 1} \frac{\beta_j(-n)}{n} x_j^n \right ) e^{\beta}\rangle\\
=& 
\langle  e^\alpha, 
\exp\left(-
\sum_{i=1}^k \sum_{n\geq 1} \frac{\alpha_i(n)}{n} w_i^n \right)\exp\left(\sum_{j=1}^\ell \sum_{n\geq 1} \frac{\beta_j(-n)}{n} x_j^n \right ) e^{\beta}\rangle\\
=& \prod_{i,j} (1-w_ix_j)^{\langle \alpha_i, \beta_j \rangle}\delta_{\alpha,-\beta}. 
\end{split}
\]
In particular, we have the following formula (see \cite[(3.1)]{ivoa}). 
\begin{lem}\label{innerprod}
For any $\alpha_1,\dots, \alpha_k, \beta_1, \dots, \beta_\ell, \alpha, \beta \in L$, we have 
\begin{equation*} 
	\begin{split}
& \sum\langle s_{\alpha_1,m_1} \dots s_{\alpha_k,m_k}e^\alpha, s_{\beta_1,n_1} \dots s_{\beta_\ell,n_\ell}e^\beta \rangle 
w_1^{m_1}\cdots w_k^{m_k} x_1^{n_1} \cdots x_\ell^{n_{\ell}}\\
=  &\prod_{i,j} (1-w_ix_j)^{\langle \alpha_i, \beta_j \rangle}\delta_{\alpha,-\beta}, 
\end{split}
\end{equation*}
where the sum is over all $m_1, \dots, m_k,n_1,\dots, n_\ell \in \ZZ_{\geq 0}$ and  $ 1\leq i\leq  k$, $ 1 \leq j \leq \ell$.  
\end{lem}

By Lemma \ref{innerprod}, $\langle u,v\rangle \in \ZZ$ for any $u,v\in V_{L, \ZZ}$. Thus, the restriction 
of $\langle \cdot, \cdot \rangle$ to $V_{L,\ZZ}$ defines  a non-degenerate $\ZZ$-valued  symmetric invariant bilinear form on $V_{L, \ZZ}$.

\subsection{Schur function} \label{sec3.1}  
Next we consider another basis of $V_{L,\ZZ}$.  
Let $\mathcal{P}$ be the set of all partitions of positive integers. Let 
$\mathcal{P}_k$ be the set of all partitions of length $\leq k$.  

\begin{de}\label{deslambda}
For  $a=(a_1,a_2, \dots, a_k)\in \ZZ_{\geq 0}^k$,  denote 
$h_a(\alpha) = s_{\alpha,a_1} \cdots s_{\alpha,a_k}$ for any $\alpha\in L$.
Set  $s_{\alpha,n}:=0 $ for any $n <0$. Then $h_a(\alpha)$ is also well-defined for any $a\in\ZZ^k$.  
Following \cite{LP,Pro} (see also \cite{ivoa}) , we define 
\begin{equation} \label{slambda0}
s_\lambda(\alpha):= \sum_{w\in Sym_k} \sgn(w) h_{\lambda+\delta-w(\delta)}(\alpha),
\end{equation} 
for $\lambda=(\lambda_1, \lambda_2, \dots, \lambda_k)\in 
\ZZ^k$, where $\sgn(w)\in \{\pm 1\}$ denotes the sign of $w$ and $\delta=\delta_k =(k-1,k-2, \dots,1, 0)$.  

Note that $\delta-w(\delta)= (1-w^{-1}(1), 2-w^{-1}(2), \dots, k-w^{-1}(k))$ for any $w\in Sym_k$. 
Then we also have 
\begin{equation} \label{slambda}
s_{\lambda} (\alpha)= \det(s_{\alpha, \lambda_i -i+j})_{1\leq i,j\leq k}. 
\end{equation}
\end{de}

\begin{rem}\label{rem3.5}
Note that $s_{\lambda} (\alpha)\in M(1)_\ZZ$ for any $\alpha\in L$  and $\lambda\in \ZZ^k$. It is proved in \cite{LP} that $s_\lambda(\alpha)=0$ unless there exist $w\in \Sym_k$ and a partition $\mu\in \mathcal{P}$ such that  $\lambda+ \delta=w(\mu+\delta)$.  Moreover,  $s_\lambda(\alpha)=\sgn(w)s_\mu(\alpha)$ if $\lambda+ \delta=w(\mu+\delta)$. 
\end{rem}

Let $x_1, x_2\dots, x_k$ be mutually commuting formal variables. For $a=(a_1, \dots, a_k)\in \ZZ^k$, we use $\x^a$ to  denote $x_1^{a_1} x_2^{a_2} \cdots x_k^{a_k}$.  Set $\Delta(x_1, \dots, x_k) : = \sum_{w\in Sym_k} \sgn(w) \x^{w(\delta)}$. The following formula is shown in \cite[Proposition 7.2]{LP}: 
 \begin{equation}\label{eq3}
	\Delta(x_1, \dots, x_k)\sum_{\lambda\in \ZZ^{k}} h_{\lambda}(\alpha) \x^\lambda =  \sum_{\lambda\in \ZZ^{k}} s_\lambda(\alpha) \x^{\lambda+\delta}. 
\end{equation}

\begin{lem}[\cite{LP,Pro}]\label{slambda1}
For any $k\in \ZZ_{>0}$, we have 
\[
\sum_{\lambda\in \ZZ^{k}} h_{\lambda}(\alpha) \x^\lambda = \prod_{1\leq i<j\leq k} (1-\frac{x_j}{x_i})^{-1} \sum_{\lambda\in \ZZ^{k}} s_\lambda(\alpha) \x^\lambda. 
\]	
\end{lem}
 
 By  Lemma \ref{innerprod}, we also have 
 \begin{equation}\label{ss}
 \begin{split}
  & \sum_{\lambda, \mu \in \ZZ^k} \langle s_\lambda(\alpha), s_\mu(\beta)  \rangle \x^\lambda \y^\mu\\
  = &  
\left[\prod_{1\leq i<j\leq k} (1-\frac{x_j}{x_i}) (1-\frac{y_j}{y_i}) \right] \prod_{1 \leq i,j \leq k}  (1- x_iy_j)^{\langle \alpha, \beta\rangle}
\end{split}
 \end{equation}
In particular, if $\langle  \alpha, \beta \rangle= -1$, then   $\langle s_\lambda(\alpha), s_\mu(\beta)  \rangle=\delta_{\lambda, \mu}$  for any $\lambda, \mu\in \mathcal{P}$. (cf.  \cite[Proposition 3.6]{ivoa}). 
 
Let $\{\gamma_1, \dots, \gamma_d\}$ be a basis of $L$ and let 
$\{\beta_1, \dots, \beta_d\} $ be the basis of $L^*$ dual to $\{\gamma_1, \dots, \gamma_d\}$. 
It is also shown in \cite{ivoa} that the set 
\[
\{s_{\lambda_1}(\gamma_1)\cdots s_{\lambda_d}(\gamma_d)e^{\alpha}\mid
\lambda_1, \dots, \lambda_d \in \mathcal{P}, \alpha\in L\}
\] 
is a basis of $V_{L,\ZZ}$. Moreover, 
the set 
\[
\{s_{\lambda_1}(-\beta_1)\cdots s_{\lambda_d}(-\beta_d)e^{-\alpha}\mid
\lambda_1, \dots, \lambda_d \in \mathcal{P}, \alpha\in L\}
\] 
is a dual basis with respect to the bilinear form $\langle \ ,\ \rangle$. That means 
\begin{equation}\label{dualBasis}
\langle s_{\lambda_1}(\gamma_1)\cdots s_{\lambda_d}(\gamma_d)e^{\alpha}, 
s_{\lambda_1'}(-\beta_1)\cdots s_{\lambda_d'}(-\beta_d)e^{\beta}\rangle 
=\delta_{(\lambda_1, \dots, \lambda_d), (\lambda_1', \dots, \lambda_d')} \delta_{-\alpha, \beta}. 
\end{equation}
We remark that  $s_{\lambda_1'}(-\beta_1)\cdots s_{\lambda_d'}(-\beta_d)e^{\beta}$ is defined on $ V_{L,\QQ}= \QQ\otimes V_{L,\ZZ}$. 

The proof of the following result can be found in \cite{Mu}. 
\begin{prop}[\cite{Mu}]\label{simple}
	Let $L$ be a positive definite even lattice. Let $V_{L,\ZZ}$  be the standard integral form of the lattice VOA $V_L$. Let $F$ be a field of characteristic $p>0$.  Then $V_{L,F}:= F\otimes V_{L,\ZZ}$ is a simple vertex operator algebra if and only if  $(p, \det(L))=1$.    
\end{prop} 

The idea of the proof in \cite{Mu}  is to show that 
\[
s_{\lambda_1}(-\beta_1)\cdots s_{\lambda_d}(-\beta_d), \quad 
\lambda_1, \dots, \lambda_d \in \mathcal{P}, 
\]
can be expressed as a  finite linear combination of polynomials of the form 
\[
s_{\alpha_1, n_1} \cdots s_{\alpha_k,n_k}, \quad \alpha_i \in \{ \gamma_1, \dots, \gamma_d\},\  n_i\in \NN
\]
with the coefficients of the form $a/b$ with $b$ relatively prime to $a$ and $p$, i.e., $(b, ap)=1$ if $(p, \det(L))=1$.  Therefore,  $s_{\lambda_1'}(-\beta_1)\cdots s_{\lambda_d'}(-\beta_d)e^{\beta}$ can be expressed as a well-defined element in $V_{L,F}$ and thus the bilinear form $\langle \ ,\ \rangle$ is non-degenerate 
on $V_{L,F}$  by \eqref{dualBasis}. 

\begin{rem}\label{rem3.9}
It is well-known \cite{D} that the inequivalent irreducible modules for $V_L$ (over $\CC$) are given by 
$V_{\beta+L}= M(1) \otimes \CC\{\beta+L\}$, where $\beta+L \in L^*/L$. 

Set $V_{\beta+L, \ZZ}$ to be  the $V_{L,\ZZ}$ submodule of $V_{\beta+L}$ generated by $e^\beta$. 
Then  $V_{\beta+L, \ZZ}$ is an integral form of $V_{\beta+L}$ for any $\beta\in L^*$. 

When $(\mathrm{char}\,F, \det(L))=1$, $V_{L,F}$ is a simple vertex operator algebra and 
$F\otimes V_{\beta+L, \ZZ}$ is an irreducible module
of $V_{L,F}$  for any $\beta\in L^*$ \cite{Mu}. 	
\end{rem}

\section{Radical of $V_{L,F}$ with $p\mid \det(L)$} 
In this section, we study the radical of the symmetric invariant bilinear form $\langle\ ,\ \rangle$ on $V_{L,F}$ if the characteristic $p$  of $F$ divides $\det(L)$.

For $i\in \NN=\{0,1,2, \dots\}$,  
set  $N_i= L\cap (p^iL^*)$. Then   for any $a\in N_i$, we have $\langle a, L\rangle \subseteq p^i \ZZ$ and \[
L=N_0 \supseteq N_1\supseteq \cdots \supseteq N_j \supseteq \cdots 
\]

\begin{lem}\label{sRi}
Let $\gamma \in N_i$. For any $n\in \ZZ$ such that $p^i\not|\,n$, we have 
\[
\langle s_{\gamma,n}, M(1)_\ZZ\rangle \equiv 0  \mod p. 
\]
\end{lem}

\begin{proof}
By Lemma \ref{innerprod},  
$\langle s_{\gamma,n}, s_{\beta_1,n_1} \dots s_{\beta_\ell,n_\ell} \rangle $
is the coefficient of $w^{n} x_1^{n_1} \cdots x_\ell^{n_{\ell}} $ in $\prod_{j=1}^\ell (1-wx_j)^{\langle \gamma, \beta_j \rangle}$. 

Since  $\gamma \in N_i$, $p^i| {\langle \gamma, \beta_j \rangle}$.  
Therefore, the coefficient of  $w^{n} x_1^{n_1} \cdots x_\ell^{n_{\ell}} $ in $\prod_{j=1}^\ell (1-wx_j)^{\langle \gamma, \beta_j \rangle}$ is divisible by $p$  unless
$p^i | n$. 
\end{proof}

\begin{lem}\label{nzero}
Let $\gamma \in N_1$. Suppose $\langle \gamma, \gamma\rangle =p^k r$ with $(p,r)=1$ and $k\ge 0$. 
Then $\langle s_{\gamma,n}, s_{\gamma,n}\rangle \not \equiv 0 \mod p$ if $n=p^k$. 
\end{lem}

\begin{proof}
Since $\langle s_{\gamma,n}, s_{\gamma,n}\rangle $
is the coefficient of $(wx)^n$ in $(1-wx)^{\langle \gamma, \gamma\rangle}$, 
we have $\langle s_{\gamma,n}, s_{\gamma,n}\rangle =-r \not \equiv 0 \mod p$ if $n=p^k$.  
\end{proof}

Since $L^*/L$ is a finite abelian group, there are positive integers  $m_1, \dots, m_r$ with 
$m_1|m_2|\cdots |m_r$ such that 
\[
L^*/L \cong \ZZ_{m_1}\times \cdots \times \ZZ_{m_r}
\]
as an abelian group.

  Let $d$ be the rank of $L$. Set $a_1=\cdots= a_{d-r}=1$ and $a_{d-r+i}=m_i$ for $i=1,\dots,r$.  Then there is a basis  $\{x_1, \dots, x_d\}$ of $L^*$ such that  $\{a_1x_1, \dots, a_d x_d\}$ is a basis of $L$.

Let $\alpha_i=a_ix_i$ and let  $\{\beta_1, \dots, \beta_d\}$ be the  basis of $L^*$ dual to $ \{ \alpha_1, \dots, \alpha_d\}$, i.e., $\langle \alpha_i,\beta_j\rangle=\langle a_i x_i ,\beta_j\rangle =\delta_{i,j}$. 
Then we also have $\langle  x_i ,a_j\beta_j\rangle =\delta_{i,j}$ for any $i,j$. 
It implies $ \{ a_1\beta_1, \dots, a_d \beta_d\}$
is dual to $\{x_1, \dots ,x_n\}$, which is also a basis of $L$. 

\begin{nota} \label{si}
Let $ \{ \alpha_1, \dots, \alpha_d\} = \{a_1x_1, \dots, a_dx_d\}$ be the basis of $L$ described above.  For each $i=1, \dots, d$, let $s_i$ be the largest non-negative integer such that $a_i=p^{s_i}r_i$.  
Set  $\gamma_i= p^{s_i} \beta_i$. Then $ \langle \alpha_i, \gamma_j\rangle = p^{s_i} \delta_{i,j} $. 
\end{nota} 

Let $M=\mathrm{Span}_\ZZ\{ \gamma_i\mid i=1, \dots n\}$. Then $L^*\supseteq  M\supseteq L$ and the index 
\[
|M/L| = |L^*/L| / |L^*/M| = r_1\cdots r_n,
\]
which is relatively  prime to $p$. 

As a consequence, we have the following result. 
\begin{lem}
There is a subset $\{\gamma_1, \dots, \gamma_n\}$ of $L^*$ such that 
\[
\langle \alpha_i, \gamma_j\rangle =p^{s_i} \delta_{i,j}\quad \text{ and }\quad  \gamma_i =\sum_{j=1}^n  \frac{c_{i,j}}{b_{i,j}} \alpha_j, 
\]
where $c_{i,j}, b_{i,j} \in \ZZ$ and $(b_{i,j}, p)=1$. 
\end{lem}

Next we will describe the radical $\mathrm{Rad}$ of $\langle\ , \ \rangle_F$  and  the quotient $V_{L,F}/\mathrm{Rad}$ using Lemmas \ref{innerprod} and \ref{sRi}.

Let  $\alpha_i, \beta_i,\gamma_i$ and $s_i$  be defined as in Notation \ref{si}. 
Let $K$ be the ideal of $M(1)_F$ generated by 
$\{ s_{\alpha_i,n}\mid n\in \ZZ_+, {p^{s_i}\not| n}, i=1,\dots,d\}$.  It is clear that $K \subseteq \mathrm{Rad}$.  
We also set    
\[
M_F = F[s_{\alpha_i, m}\mid    m\in \ZZ_+,\  p^{s_i}| m \text{ and }i=1, \dots, n]
\] 
be the set of polynomials generated by  $s_{\alpha_i, m}$, where $m\in \ZZ_+$ and $m$ is divisible by $p^{s_i}$. Then $M(1)_F = M_F \oplus K$. 

\begin{nota}\label{PM}
Let $P_M:M(1)_F \to M_F$ be the natural projection. Then for any $u\in M(1)_F$, 
$u-P_M(u)\in K \subset \mathrm{Rad}$.  
\end{nota}

Denote 
\[
U_F = M_F \otimes F\{L\} \subset  V_{L,F}. 
\]
Then $V_{L,F}=U_F \oplus K\otimes F\{L\}$ as a vector space. Note also that $K\otimes F\{L\} \subseteq \mathrm{Rad}$. 

The following is the main theorem of this article. 

\begin{thm}\label{main}
Let $L$ be a positive definite even lattice and let $F$ be a field of characteristic $p>0$. 
Then  the bilinear form on $U_F$ is non-degenerate and $\mathrm{Rad}=K\otimes F\{L\}$. As a vector space, 
\[
V_{L,F}/ Rad\cong U_F\cong  F[ s_{\alpha_i, p^{s_i} n}\mid  n\in \ZZ_+, i=1, \dots,d] \otimes F\{L\}. 
\]
\end{thm}

We will prove the theorem above using a strategy  as in \cite{Mu}. 
Recall from Definition \ref{deslambda} that $s_{\lambda} (\alpha)= \det(s_{\alpha, \lambda_i -i+j})_{1\leq i,j\leq k}$. We consider a generalization as follows. 

\begin{de}\label{de4.7}
Let $p$ be a prime and let $k$ be a positive integer.  For $s\in \NN$, $\alpha\in L^*$ and $\lambda=(\lambda_1, \dots, \lambda_k)\in \ZZ^k$,  
 we define 
\[
s^{(s)}_\lambda( \alpha) :=  \det(s_{\alpha, p^{s}(\lambda_i -i+j)})_{1\leq i,j\leq k}
= \sum_{w\in Sym_k} \sgn(w) h_{p^s(\lambda+\delta-w(\delta))}(\alpha).
\]
\end{de}

By Lemma \ref{slambda1}, we have 
\begin{equation} \label{form5}
\sum_{\lambda\in \ZZ^{k}} h_{p^s\lambda}(\alpha) \x^{p^s\lambda} = \prod_{1\leq i<j\leq k} (1-(\frac{x_j}{x_i})^{p^s})^{-1}\sum_{\lambda\in \ZZ^k} s^{(s)}_\lambda(\alpha) \x^{p^s\lambda} .   
\end{equation}

Let $\alpha\in N_s \setminus N_{s+1}$. Then there is a $\beta\in L$ such that 
$p^{s}| \langle \alpha, \beta \rangle$ but 
$p^{s+1}\not| \langle \alpha, \beta \rangle$. In this case,  $\langle \alpha, \beta \rangle=p^sr$ with $(p,r)=1$.  Then we have
\[
\sum \langle h_\lambda(\alpha), h_\mu(\beta)\rangle \x^\lambda \y^\mu 
= \prod_{1\leq i, j\leq k} (1-x_iy_j)^ {\langle \alpha, \beta \rangle}
=  \prod_{1\leq i, j\leq k} (1-(x_iy_j)^{p^s})^{r} \mod p.  
\]

By Lemma \ref{sRi}, if $\lambda=(\lambda_1, \dots, \lambda_k)$ and ${p^s\not| \lambda_i}$ for some $i$, then we have $\langle h_\lambda(\alpha), h_\mu(\beta)\rangle\equiv 0 \mod p$ for any $\mu\in \ZZ^k$. Therefore, we also have
\[
\sum \langle h_{p^s\lambda}(\alpha), h_{p^s\mu}(\beta)\rangle \x^{p^s\lambda} \y^{p^s\mu} 
=  \prod_{1\leq i, j\leq k} (1-(x_iy_j)^{p^s})^{r} \mod p.  
\]
By the same argument for \eqref{ss},  we have 
 \begin{equation}\label{ssr}
 \begin{split}
  & \sum_{\lambda, \mu} \langle s^{(s)}_\lambda(\alpha), s^{(s)}_\mu(\beta)  \rangle \x^{p^s\lambda} \y^{p^s\mu}\\
  = &  
\left[\prod_{1\leq i<j\leq k} \left(1-(\frac{x_j}{x_i})^{p^s}\right) 
\left(1-(\frac{y_j}{y_i})^{p^s}\right) \right] \prod_{1 \leq i,j \leq k}  (1- (x_iy_j)^{p^s})^{r}. 
\end{split}
 \end{equation}

\medskip

Let $\alpha_i, x_i, a_i$ and $s_i$, $i=1, \dots,n$  be as in Notation  \ref{si}. 
 
\begin{lem}
The set 
\[
s^{(s_1)}_{\lambda_1}(\alpha_1) \cdots s^{(s_n)}_{\lambda_n}(\alpha_n) e^\alpha
\]
for $\lambda_1, \dots, \lambda_n\in \mathcal{P}$ and $\alpha\in L$ forms an $F$-basis 
of $U_F$. 
\end{lem}

\begin{proof}
By definition, it is clear that 
$s^{(s_1)}_{\lambda_1}(\alpha_1) \cdots s^{(s_n)}_{\lambda_n}(\alpha_n) e^\alpha$ is in 
$U_F$. By Formula \eqref{form5}, it is also clear that $s_{\alpha_i, m_1} \dots s_{\alpha_i,m_k}$, where  $m_1, \dots, m_k$ are divisible by $p^{s_i}$,  is a linear combination of $s^{(s_i)}_{\alpha_i}(\lambda)$. Thus,
\[
\{ s^{(s_1)}_{\lambda_1}(\alpha_1) \cdots s^{(s_n)}_{\lambda_n}(\alpha_n) e^\alpha
\mid \lambda_1, \dots, \lambda_n\in \mathcal{P}, \alpha\in L\} 
\] 
spans $U_F$. The linear independence follows from \cite[Proposition 3.6]{ivoa}. 
\end{proof}

By \eqref{ssr} and the same argument as in \cite[Proposition 3.6]{ivoa}, we have the following result. 

\begin{lem}
For $\lambda, \lambda' \in \mathcal{P}$ and $\alpha, \gamma\in L$,
\[
\langle s_{\lambda_1}^{(s_1)}(\alpha_1)\cdots s_{\lambda_d}^{(s_d)}(\alpha_d)e^{\alpha}, s_{\lambda_1'}^{(s_1)}(\gamma_1)\cdots s_{\lambda_d'}^{(s_d)}(\gamma_d)e^{\gamma} 
\rangle 
=\delta_{(\lambda_1, \dots, \lambda_d), (\lambda_1', \dots, \lambda_d')} \delta_{-\alpha, \gamma}. 
\] 
\end{lem}

Next we recall several results from \cite{Mu}.
\begin{lem}
Suppose that $a$ is an integer and $b$ is a positive integer which is relatively
prime to $a$ and $p$. Then $\binom{a/b}{n}= (a/b)(a/b-1) \cdots (a/b-n+1)/ n!$ is a well defined element in $F$ for any non-negative integer $n$. 
\end{lem}

\begin{lem}
Suppose that $a$ is an integer and $b$ is a positive integer which is relatively
prime to $a$ and $p$. For every $n\in \NN$, define $s_{\frac{a}b, n}$ by\
\[
\sum_{n\geq 0} s_{\frac{a}b \alpha,n} z^n =\left( 1+ \sum_{n> 0} s_{\alpha,n} z^n\right )^{\frac{a}b}.  
\] 
Then $s_{\frac{a}b \alpha,n}$ is a finite sum of elements of the form 
$\binom{{a}/b}{k} s_{\alpha, m_1}\cdots s_{\alpha, m_k}$, where $k\geq 0$, $m_1, \dots, m_k>0$.   
\end{lem}

\noindent \textbf{Proof of Theorem \ref{main}}

By the definition of  $\alpha_i$ and $\gamma_i$ (see Notation \ref{si}), for each $i$, 
\[
\gamma_i =\sum_{j=1}^n  \frac{c_{i,j}}{b_{i,j}} \alpha_j, 
\]
where $c_{i,j}, b_{i,j} \in \ZZ$ and $(b_{i,j}, p)=1$. 
Then  we have 
\[
\sum_{k\geq 0}  s_{\gamma_i,k} z^k  = \prod_{j=1}^n \left( \sum_{m\geq 0}  s_{\alpha_j,m}z^m \right)^{\frac{c_{i,j}}{b_{i,j}}}
\] 
Therefore, $s_{\lambda_i'}^{(s_i)}(\gamma_i)$ can be written as a linear combination of $s_{\beta_1, n_1}\cdots s_{\beta_k, n_k}$ with $\beta_i\in \{\alpha_1, \dots, \alpha_n\}$ and every coefficient has the form $a/b$ with $a,b\in \ZZ$ and $(p,b)=1$. 
Therefore,  $s_{\lambda_1'}^{(s_1)}(\gamma_1)\cdots s_{\lambda_d'}^{(s_d)}(\gamma_d)$ is a well defined element in $M(1)_F$. Set $u= P_M(s_{\lambda_1'}^{(s_1)}(\gamma_1)\cdots s_{\lambda_d'}^{(s_d)}(\gamma_d))$ (see Notation \ref{PM}).  Then  $u\otimes e^{\gamma}\in U_F$ and 
\[
\begin{split}
&\langle s_{\lambda_1}^{(s_1)}(\alpha_1)\cdots s_{\lambda_d}^{(s_d)}(\alpha_d)e^{\alpha}, u\otimes e^{\gamma}\rangle \\
= & 
\langle s_{\lambda_1}^{(s_1)}(\alpha_1)\cdots s_{\lambda_d}^{(s_d)}(\alpha_d)e^{\alpha}, s_{\lambda_1'}^{(s_1)}(\gamma_1)\cdots s_{\lambda_d'}^{(s_d)}(\gamma_d)e^{\gamma} 
\rangle \\
=&\delta_{(\lambda_1, \dots, \lambda_d), (\lambda_1', \dots, \lambda_d')} \delta_{-\alpha, \gamma}. 
\end{split}
\]
Therefore, the bilinear form $\langle \ ,\ \rangle_F$ is non-degenerate on $U_F=M_F \otimes F\{L\}$ and 
$\mathrm{Rad}=K\otimes F\{L\}$. 
\qed

\subsection{Some examples}
Next we study several explicit examples. 

\begin{de}
Let $V=\oplus_{n=0}^\infty V_n $ be a $\ZZ$-graded vertex algebra  over a field $F$ such that $\dim_F V_n < \infty$ for all $n\geq 0$. Then the character of $V$ is defined to be the $q$-series
\[
ch\, V:= \sum_{n=0}^\infty \dim_F(V_n) q^n.
\]
\end{de}

\begin{ex}
Let $L= A_1$ be the root lattice of type $A_1$, i.e., $L=\ZZ\alpha$ with $\langle \alpha, \alpha\rangle=2$. Set $F=\mathbb{F}_2$. 
Then the radical $\mathrm{Rad}$ of $\langle \ ,\ \rangle$ on $V_{L,F}$ is generated 
by  the set of $s_{\alpha,2n+1}$ for all $n\geq 0$, i.e., 
 \[
 \mathrm{Rad} =\mathrm{Span}_F\{ s_{\alpha,2n_1+1}\cdots s_{\alpha,2n_i+1}\mid n_1,  \dots, n_i\geq 0, i>0\}
 \]
The quotient $V_{L,F}/ \mathrm{Rad} \cong F[s_{\alpha, 2n} \mid n\in \ZZ_+] \otimes F\{L\}$ as a vector space.   
In this case, we have 
\[
ch_{V_{L,F}/ \mathrm{Rad}} (q) = \frac{\sum_{n=0}^{\infty} q^{n^2}} {\prod_{n=1}^\infty
(1-q^{2n})}. 
\] 

\end{ex}

\begin{ex}\label{eA2}
Let $L= A_2$ be the root lattice of type $A_2$ and let $F=\mathbb{F}_3$ be a field with $3$ elements. Let $\{\beta_1, \beta_2\}$  be a basis of $L$ such that $\langle \beta_1, \beta_1\rangle=\langle \beta_2, \beta_2\rangle=2$ and $\langle \beta_1, \beta_2\rangle= -1$. So $\det(L)=3$. 

Set $\alpha_1= \beta_1$ and $\alpha_2= \beta_1- \beta_2$. Then $\{\alpha_1, \alpha_2\}$ is also a basis of $L$ and $\langle \alpha_2, L\rangle=3\ZZ$.   
Then 
\[
V_{L,F}/ \mathrm{Rad} \cong F[s_{\alpha_1, n}| n\in \ZZ_+]\otimes F[s_{\alpha_2, 3n}| n\in \ZZ_+]\otimes F\{A_2\}. 
\]
The character  is given by 
\[
ch_{V_{L,F}/ \mathrm{Rad}} (q) = \frac{\theta_{A_2} (q)} {\prod_{n=1}^\infty
(1-q^{n}) (1-q^{3n})}, 
\] 
where $\theta_{A_2}(q) =\sum_{\alpha\in A_2} q^{\langle \alpha, \alpha\rangle /2}$ is the theta series of the lattice $A_2$.  
\end{ex}

\begin{rem}Let $\beta \in L^*$.  
As in Remark \ref{rem3.9},  the $V_{L,\ZZ}$ submodule $V_{\beta+L, \ZZ}$  of $V_{\beta+L}$ generated by $e^\beta$ is an integral form of $V_{\beta+L}$ and $F\otimes V_{\beta+L, \ZZ}$ is a module of $V_{L,F}$  for any commutative ring  $F$.

Let $p$ be a prime and set $L=A_{p-1}$ and $F=\FF_p$. 
Let $\{\alpha_1,\dots  ,\alpha_{p-1}\}$ be a set of simple roots of $A_{p-1}$ and let 
$\gamma_1, \dots, \gamma_{p-1}$ be the fundamental weights, i.e., $\langle \gamma_i, \alpha_i\rangle =1$ and $\langle \gamma_i, \alpha_j\rangle =0$ for $j\neq i$.  Then 
\[
A_{p-1}^*/A_{p-1} = A_{p-1}\cup (\gamma_1+A_{p-1}) \cup \cdots \cup (\gamma_{p-1} +A_{p-1}) \cong \ZZ_p
\]  
and $\langle \gamma_i, \gamma_i\rangle = i(p-i)/p$ for $i=1, \dots, p-1$. 
Let $\beta_i=p\gamma_i$. Then $\beta_i\in A_{p-1} $ and $\langle \beta_i ,A_{p-1}\rangle \subset p\ZZ$. Therefore,  $\beta_i(-1)\cdot\mathbf{1} =s_{\beta_i,1}\in \mathrm{Rad}_{\langle \ , \ \rangle_{V_{A_{p-1},F}}}$ but 
\[
\beta_i(0) e^{\gamma_i} = \langle \beta_i, \gamma_i\rangle e^{\gamma_i} =i(p-i) e^{ \gamma_i}\neq 0 \quad \text{ for  } 1\leq i\leq p-1.
\] 
Since $\mathrm{Rad}$ does not annihilate  $e^{\gamma_i}$,  $V_{\gamma_i+ A_{p-1}, F} $ is not a module for $V_{A_{p-1},F}/\mathrm{Rad}$.    
\end{rem}

\end{document}